\documentclass[11pt]{article}
\usepackage{enumerate}
\usepackage{amssymb,a4wide,latexsym,makeidx,epsfig,fleqn}
\usepackage{amsthm}
\usepackage{amsmath}
\usepackage{enumerate}
\newtheorem{theorem}{Theorem}[section]
\newtheorem{lemma}{Lemma}[section]

\newtheorem{corollary}{Corollary}[section]
\newtheorem{claim}{Claim}[section]

\begin{document}
\textwidth 150mm \textheight 225mm
\title{Tur\'{a}n numbers of general hypergraph star forests \thanks{Supported by the National Natural Science Foundation of
China (No.12271439) and China Scholarship Council (No. 202206290003).}}
\author{{Lin-Peng Zhang\textsuperscript{a,b}, Hajo Broersma\textsuperscript{b}\footnote{Corresponding author.}, Ligong Wang\textsuperscript{a}}\\
{\small \textsuperscript{a} School of Mathematics and Statistics,}\\
{\small Northwestern Polytechnical University, Xi'an, Shaanxi 710129, P.R. China.}\\
{\small \textsuperscript{b} Faculty of Electrical Engineering, Mathematics and Computer Science,}\\
{\small  University of Twente, P.O. Box 217, 7500 AE Enschede, the Netherlands}\\
{\small E-mail: lpzhangmath@163.com, h.j.broersma@utwente.nl, lgwangmath@163.com}}
\date{}
\maketitle
\begin{center}
\begin{minipage}{135mm}
\vskip 0.3cm
\begin{center}
{\small {\bf Abstract}}
\end{center}
{\small Let $\mathcal{F}$ be a family of $r$-uniform hypergraphs,  and let $H$ be an $r$-uniform hypergraph. 
Then $H$ is called $\mathcal{F}$-free if it does not contain any member of 
 $\mathcal{F}$ as a subhypergraph. The Tur\'{a}n number of $\mathcal{F}$, denoted by $ex_r(n,\mathcal{F})$, is the maximum number of hyperedges in an $\mathcal{F}$-free $n$-vertex $r$-uniform hypergraph. 
Our current results are motivated by earlier results on Tur\'{a}n numbers of star forests and hypergraph star forests.  
In particular, Lidick\'{y}, Liu and Palmer [Electron. J. Combin. 20 (2013)] determined the Tur\'{a}n number $ex(n,F)$ of a star forest $F$ for sufficiently large $n$. 
Recently, Khormali and Palmer [European. J. Combin. 102 (2022) 103506] generalized the above result to three different well-studied hypergraph settings, but restricted to the case that all stars in the hypergraph star forests are identical. We further generalize these results to general hypergraph star forests. 
 }
\end{minipage}
\end{center}

\section{Introduction}
The results in this paper deal with Tur\'{a}n numbers of hypergraph star forests. Our work is motivated by recent results by different groups of researchers, and aims at generalizing their results. Before we can give more details, we need to introduce and recall some essential terminology and notation. 

Let $\mathcal{F}$ be a family of $r$-uniform hypergraphs. An $r$-uniform hypergraph $H$ is called $\mathcal{F}$-free if it does not contain any 
member of $\mathcal{F}$ as its subhypergraph. Note that here we mean subhypergraph, and not necessarily induced subhypergraph. 
The Tur\'{a}n number of $\mathcal{F}$, denoted by $ex_r(n,\mathcal{F})$, is the maximum number of hyperedges in an $r$-uniform $n$-vertex $\mathcal{F}$-free hypergraph. When $r=2$, instead of 
writing $ex_2(n,\mathcal{G})$ we use $ex(n,\mathcal{G})$ for a family of graphs $\mathcal{G}$. 
Determining the Tur\'{a}n numbers for specific graphs and graph classes is a classic and well-studied area within graph theory. This started already back in 1907, when Mantel~\cite{M07} determined the Tur\'{a}n number of a triangle $K_3$. The term Tur\'{a}n number reflects the fact that in 1941, Tur\'{a}n~\cite{T41} generalized Mantel's result by his seminal theorem that settles the problem of determining the Tur\'{a}n number of a complete graph $K_k$ for general $k$. Since then, there have appeared many papers and monographs on this type of problems. One of the most fundamental results within this area of research is the Erd\H{o}s-Stone-Simonovits theorem, which gives the following asymptotic result for all $k$-chromatic graphs $F$:
\begin{equation*}
ex(n,F)=\bigg(1-\frac{1}{k-1}\bigg)\frac{n^2}{2}+o(n^2).
\end{equation*}
Here $k$ denotes the chromatic number of $F$, {\it i.e.}, the smallest number of colors needed to assign one color to each vertex of $F$ in such a way that adjacent vertices receive different colors. 
Note that the above formula only gives $ex(n,F)=o(n^2)$ for a bipartite (2-chromatic) graph $F$. Determining Tur\'{a}n numbers for bipartite graphs remains an active area of research. We refer to the survey of F\"{u}redi and Simonovits~\cite{FS13} for an extensive history.

In this paper, we focus on sparse bipartite graphs, namely star forests, and their hypergraph counterparts. 
Fix a positive integer $\ell$. Let  $S_{\ell}$ denote the star with $\ell$ edges, {\it i.e.}, the complete bipartite graph $K_{1,\ell}$. 
A star forest is a graph consisting of a number of vertex-disjoint stars. Let $G$ and $H$ be two vertex-disjoint graphs. The union of $G$ and $H$, denoted 
by $G\cup H$, is the graph with vertex set $V(G)\cup V(H)$ and edge set $E(G)\cup E(H)$. Denote by 
$k\cdot G$ the graph composed of $k$ pairwise vertex-disjoint copies of the graph $G$.

In 2013, Lidick\'{y}, Liu and Palmer~\cite{LLP13} determined the Tur\'{a}n number of a star forest for sufficiently large $n$. 
\begin{theorem}[Lidick\'{y}, Liu and Palmer, \cite{LLP13}]\label{ch1:LLP}
Let $d_1\ge d_2\ge \cdots \ge d_k$ be $k$ positive integers. For $n$ sufficiently large,
$$
ex\bigg(n,\bigcup\limits_{i=1}^k S_{d_i}\bigg)=\mathop{max}\limits_{1\le i\le k}\Bigg\{(i-1)(n-i+1)+\binom{i-1}{2}+\left\lfloor \frac{d_i-1}{2}(n-i+1)\right\rfloor\Bigg\}.
$$
\end{theorem}

We come back to this later, but first continue with a discussion on hypergraph analogues of  Tur\'{a}n type extremal problems. 
Determining Tur\'{a}n numbers for hypergraphs is another direction which has attracted a lot of attention. 
One of the issues that we meet is to give a suitable and reasonable generalization of graph structures to hypergraph structures. 
It is natural to generalize the notion of a complete graph to a complete hypergraph in the following way. 
Since the edge set of a complete graph $K_k$ on $k$ vertices is the set of all possible 2-element subsets of these $k$ vertices, the set of hyperedges of a  
complete $r$-uniform $k$-vertex hypergraph is the set of all possible $r$-element subsets of these $k$ vertices.  
The difficulty of determining the Tur\'{a}n numbers for hypergraphs is reflected by the fact that we still do not 
know the Tur\'{a}n number of $K^3_4$, the complete 3-uniform 4-vertex hypergraph.

In this paper, we aim to determine Tur\'{a}n numbers for analogues of star forests in the hypergraph setting. 
For this, we will in fact consider three different hypergraph settings that are known from literature. The first hypergraph setting we consider is based on the notion of expansion. 

Fix a graph $F$ and an integer $r\ge 2$. The expansion of $F$ is the $r$-uniform hypergraph $F^+$ constructed by adding $r-2$ new distinct vertices to each edge of $F$, where all added new vertices are distinct. When $r=2$, the expansion $F^+$ is exactly the original graph $F$. There already exist several papers involving Tur\'{a}n numbers for various expansions (See, {\it e.g.}, \cite{BK14,F14,FJ14,FJS14,KMV15-1,KMV15-2,KMV17,M06,P13}). 
The case with $F^+=S^+_{\ell}$ has a considerable history. See the survey~\cite{MV11} for an overview.

In 1977, Duke and Erd\H{o}s~\cite{DE97} proved the following theorem. 
\begin{theorem}[Duke and Erd\H{o}s, \cite{DE97}]\label{ch1:DE}
Fix integers $\ell\ge 2$ and $r\ge 3$. Then there exists a constant $c(r)$ such that for sufficiently large $n$, 
$$
ex_r(n,S^+_{\ell})\le c(r)\ell(\ell-1)n^{r-2}.
$$
\end{theorem}
We can construct an extremal hypergraph with size $ex_r(n,S^+_{\ell})$ of order $n^{r-2}$ as follows. 
Take an $r$-uniform hypergraph all hyperedges of which contain a fixed pair of vertices.

For presenting the next result due to Erd\H{o}s~\cite{E65}, let $M_k$ be a matching of size $k$, {\it i.e.}, 
a set of $k$ disjoint edges. We can view $M_k$ as a star forest each star of which is $S_1$, an edge. The following well-known theorem 
gives the Tur\'{a}n number of the expansion of $M_k$, when the number of vertices is sufficiently large.
\begin{theorem}[Erd\H{o}s, \cite{E65}]\label{ch1:E}
Fix integers $k\ge 1$ and $r\ge 2$. Then for sufficiently large $n$, 
$$
ex_r(n,M^+_{k})=\binom{n}{r}-\binom{n-k+1}{r}.
$$
\end{theorem}

In a recently published paper, Khormali and Palmer~\cite{KP22} generalized the above result to the expansion $k\cdot S^+_\ell$ of a star forest $k\cdot S_\ell$ composed 
of $k$ copies of the star $S_{\ell}$.
\begin{theorem}[Khormali and Palmer, \cite{KP22}]\label{ch1:KP1}
Fix integers $\ell,k\ge 1$ and $r\ge 2$. Then for sufficiently large $n$,
$$
ex_r(n,k\cdot S^+_\ell)=\binom{n}{r}-\binom{n-k+1}{r}+ex_r(n-k+1, S^+_\ell).
$$
\end{theorem}

In our first new main result, we generalize the above result by obtaining the Tur\'{a}n number of 
the expansion of a general star forest. We postpone our proof of the next result to Section~\ref{sec-thm1}. 
\begin{theorem}\label{ch1:thm1}
Fix integers $d_1\ge d_2\ge \cdots \ge d_k$, $k\ge 1$ and $r\ge 2$. Then for 
sufficiently large~$n$, 
   \begin{equation*}
     ex_r\bigg(n,\bigcup\limits_{i=1}^k S^+_{d_i}\bigg)=\mathop{max}\limits_{1\le i\le k}\Bigg\{\binom{n}{r}-\binom{n-i+1}{r}+ex_r\Big(n-i+1,S^+_{d_i}\Big)\Bigg\}.
   \end{equation*}
 \end{theorem}

For our next main result, we turn to another hypergraph setting, which is based on the notion of a linear hypergraph. 
An $r$-uniform hypergraph is called linear if every pair of hyperedges of it shares at most one vertex.
It is easy to see that each expansion is a linear hypergraph. Hence, we can consider corresponding problems among linear hypergraphs. Let $F$ be a linear $r$-uniform hypergraph. 
Then the linear Tur\'{a}n number of $F$, denoted by $ex^{lin}_r(n,F)$, is the maximum number of hyperedges in an $r$-uniform $n$-vertex $F$-free linear hypergraph. 
A seminal result in this context settles the problem of determining the linear Tur\'{a}n number $ex^{lin}_3(n,C^+_3)$ of 
the expansion of $C_3$, which is equivalent to the famous (6,3)-problem. 
Motivated by this result, more and more papers on this topic are starting to emerge (See, {\it e.g.}, \cite{CGJ18, EGM19, EGM21, FG20, GCS23, GMV19, GRS22, GS21, LV03, S23, TWZZ22, CT17}). 

In 2022, Khormali and Palmer \cite{KP22} determined the linear Tur\'{a}n number of $k\cdot S^+_l$ asymptotically for sufficiently large $n$.
 \begin{theorem}[Khormali and Palmer, \cite{KP22}]\label{ch1:KP2}
 Fix integers $\ell,k\ge 1$ and $r\ge 2$. Then for sufficiently large $n$,
 $$
 ex^{lin}_r(n,k\cdot S^+_\ell)\le \bigg(\frac{\ell-1}{r}+\frac{k-1}{r-1}\bigg)(n-k+1)+\frac{\binom{k-1}{2}}{\binom{r}{2}}.
 $$
 Moreover, this bound is sharp asymptotically.
 \end{theorem}

In our second main result, we determine the linear Tur\'{a}n number of a general hypergraph star forest asymptotically for sufficiently large $n$. 
 
 \begin{theorem}\label{ch1:thm2}
Fix integers $d_1\ge d_2\ge \cdots\ge d_k$, $k\ge 1$ and $r\ge 2$. Then for sufficiently large $n$, 
\begin{equation*}
ex^{lin}_r\bigg(n,\bigcup\limits_{i=1}^k S^+_{d_i}\bigg)\le \mathop{max}\limits_{1\le i\le k}\bigg\{\bigg(\frac{d_i-1}{r}+\frac{i-1}{r-1}\bigg)(n-i+1)+\frac{\binom{i-1}{2}}{\binom{r}{2}}\bigg\}.
\end{equation*}
Moreover, this bound is sharp asymptotically.
\end{theorem}

We postpone our proof of the above result to Section~\ref{sec-thm2}. 
In the following, we turn to a third hypergraph setting, as defined by Gerbner and Palmer~\cite{GP17}. Let $F$ be a graph. An $r$-uniform hypergraph $\mathcal{H}$ is called  
a Berge-$F$ if there is an injection $f: V(F)\rightarrow V(\mathcal{H})$ 
and a bijection $f':E(F)\rightarrow E(\mathcal{H})$ such that for every edge $uv\in E(F)$ we have $\{f(u),f(v)\}\subseteq f'(uv)$. For a fixed graph $F$, we can 
observe that there are many different hypergraphs that are a Berge-$F$, and that 
a fixed hypergraph $\mathcal{H}$ can be a Berge-$F$ for more than one graph $F$. For a short survey of extremal results for Berge-$F$-free hypergraphs see Subsection 5.2.2 in~\cite{GP18}.

In 2022, by considering distinct values of $r,k,\ell$ in the following two results, 
Khormali and Palmer \cite{KP22} determined the Tur\'{a}n number of a Berge-$k\cdot S_\ell$ 
 asymptotically for sufficiently large $n$. 
 
\begin{theorem}[Khormali and Palmer, \cite{KP22}]
Fix integers $\ell,k\ge 2$ and $r\ge \ell+k-1$. Then for sufficiently large $n$, 
\begin{equation*}
ex_r(n,\mbox{Berge-}k\cdot S_\ell)\le \frac{\ell-1}{r-k+1}(n-k+1).
\end{equation*}
Moreover, this bound is sharp whenever $r-k+1$ divides $n-k+1$.
\end{theorem}

\begin{theorem}[Khormali and Palmer, \cite{KP22}]
Fix integers $\ell,k\ge 2$ and $r\le \ell+k-2$. Then for sufficiently large $n$, 
\begin{equation*}
ex_r(n,\mbox{Berge-}k\cdot S_\ell)\le \bigg(\binom{\ell +k-1}{r}-\binom{k-1}{r}\bigg)\left\lceil\frac{n-k+1}{\ell}\right\rceil+\binom{k-1}{r}.
\end{equation*}
Moreover, this bound is sharp whenever $\ell$ divides $n-k+1$.
\end{theorem}

With our final two main results, we generalize 
the above results to more general Berge-star forests. We postpone our proofs to Section~\ref{sec-thm3} and \ref{sec-thm4}, respectively. 

\begin{theorem}\label{ch1:thm3}
Fix integers $d_1\ge d_2\ge \cdots \ge d_k\ge 2$, $k\ge 2$ and $r\ge d_1+k-1$. Then 
for sufficiently large $n$, 
\begin{equation*}
ex_r\bigg(n,\mbox{Berge-}\bigcup\limits_{i=1}^k S_{d_i}\bigg)\le \mathop{max}\limits_{1\le s\le k-1} \bigg\{\frac{d_{s+1}-1}{r-s}(n-s)\bigg\}.
\end{equation*}
Moreover, this bound is sharp whenever $r-s$ divides $n-s$ for $1\le s\le k-1$.
\end{theorem}

\begin{theorem}\label{ch1:thm4}
Fix integers $d_1\ge d_2\ge \cdots \ge d_k\ge 2$, $k\ge 2$ and $r\le d_1+k-2$. Then 
for sufficiently large $n$, 
\begin{equation*}
\begin{aligned}
ex_r\bigg(n,\mbox{Berge-}\bigcup\limits_{i=1}^k S_{d_i}\bigg)&\le \mathop{max}\bigg\{\mathop{max}\limits_{1\le i\le k}\bigg\{\bigg(\binom{d_i+i-1}{r}-\binom{i-1}{r}\bigg)\left\lceil\frac{n-i+1}{d_i}\right\rceil+\binom{i-1}{r},\\
&\frac{d_i-1}{r-i+1}(n-i+1)\bigg\},\mathop{max}\limits_{1\le c<k-1}\bigg\{\frac{\mathop{max}\Big\{\binom{d_c+c}{r-1},d_c+c\Big\}}{r-c}(n-c)\bigg\}\bigg\}.
\end{aligned}
\end{equation*}
Moreover, this bound is sharp whenever $d_i$ and $r-i+1$ divide $n-i+1$ for $1\le i\le k$.
\end{theorem}

The remainder of this paper is devoted to the proofs of our new contributions. We complete the paper by a short concluding section. 

\section{Proof of Theorem \ref{ch1:thm1}}\label{sec-thm1}
By Theorems \ref{ch1:LLP} and \ref{ch1:E}, the theorem holds for $r=2$ or $d_1=1$. Thus, 
it suffices to show that the theorem holds when $r\ge 3$ and $d_1\ge 2$.

For the lower bound, we can construct an extremal hypergraph $H$ as follows. Fix $i$ such that $1\le i\le k$, and let $A$ and $B$ be vertex sets of $i-1$ and $n-i+1$ vertices, respectively. Firstly, we embed an $S^+_{d_i}$-free $r$-uniform hypergraph with $ex_r(n-i+1,S^+_{d_i})$ hyperedges into $B$. Next, we add every $r$-set that is incident to $A$ to our hypergraph. Since $B$ contains no copy of $S^+_{d_i}$, each $S^+_{d_j}$ with $1\le j\le i$ must contain at least one vertex of $A$.
Therefore, $H$ is $\bigcup\limits_{i=1}^k S^+_{d_i}$-free. Thus, we can choose $i$ such that the number of hyperedges in $H$ 
attains the value in the statement of the theorem.

In the following, we consider the upper bound. Let $\mathcal{H}$ be an $r$-uniform $n$-vertex hypergraph with
\begin{equation*}
|E(\mathcal{H})|>\mathop{max}\limits_{1\le i\le k}\Bigg\{\binom{n}{r}-\binom{n-i+1}{r}+ex_r\Big(n-i+1,S^+_{d_i}\Big)\Bigg\}. 
\end{equation*}
We need to show that $\mathcal{H}$ is not $\bigcup\limits_{i=1}^k S^+_{d_i}$-free. We proceed by induction on $k$. For $k=1$, since $|E(\mathcal{H})|>ex_r(n,S^+_{d_1})$ we have $\mathcal{H}$ is not $S^+_{d_1}$-free. Let $k>1$ and assume the statement holds for $k-1$. In the following, we discuss two cases based on the maximum degree $\Delta(\mathcal{H})$ of $\mathcal{H}$.

\medskip\noindent
{\bf Case 1:}
\begin{equation*}
\Delta(\mathcal{H})<\frac{1}{\bigcup\limits_{i=1}^{k-1} ((r-1)d_i+1)}\bigg(\binom{n}{r}-\binom{n-k+1}{r}\bigg).
\end{equation*}

Assume that $\mathcal{H}$ contains a copy of $\cup_{i=1}^t S^+_{d_i}$ such that $t$ is maximal.
We claim that $t\ge k$. Otherwise, there are at most $\cup_{i=1}^{k-1} ((r-1)d_i+1)$ vertices spanned by
 $\cup_{i=1}^t S^+_{d_i}$. After removing these vertices and the incident hyperedges we have at least
\begin{equation*}
\begin{aligned}
|E(\mathcal{H})|&-\bigcup\limits_{i=1}^{k-1} ((r-1)d_i+1)\cdot \Delta(\mathcal{H})>|E(\mathcal{H})|-\bigg(\binom{n}{r}-\binom{n-k+1}{r}\bigg)\\
&>\binom{n}{r}-\binom{n-k+1}{r}+ex_r\Big(n-k+1,S^+_{d_k}\Big)-\bigg(\binom{n}{r}-\binom{n-k+1}{r}\bigg)\\
&= ex_r\Big(n-k+1,S^+_{d_k}\Big)
\end{aligned}
\end{equation*}
hyperedges. Therefore, there is a copy of $S^+_{d_k}$ that is vertex-disjoint from 
 $\bigcup\limits_{i=1}^t S^+_{d_i}$. This violates the maximality of $t$, a contradiction.

\medskip\noindent
{\bf Case 2.}
\begin{equation*}
\Delta(\mathcal{H})\ge\frac{1}{\bigcup\limits_{i=1}^{k-1} ((r-1)d_i+1)}\bigg(\binom{n}{r}-\binom{n-k+1}{r}\bigg).
\end{equation*}
Let $u$ be a vertex of maximum degree. Since $d(u)\le \binom{n-1}{r-1}$,
\begin{equation*}
\begin{aligned}
|E(\mathcal{H})|-d(u)&>\mathop{max}\limits_{1\le i\le k}\bigg\{\binom{n}{r}-\binom{n-i+1}{r}+ex_r\Big(n-i+1, S^+_{d_i}\Big)\bigg\}-\binom{n-1}{r-1}\\
&=\mathop{max}\limits_{1\le i\le k}\bigg\{\binom{n-1}{r}-\binom{n-(i-1)}{r}+ex_r\Big(n-(i-1), S^+_{d_i}\Big)\bigg\}\\
&\overset{d'_j=d_{j+1}}{\ge} \mathop{max}\limits_{1\le j\le k-1}\bigg\{\binom{n-1}{r}-\binom{n-j}{r}+ex_r\Big(n-j, S^+_{d'_j}\Big)\bigg\}\\
&=ex_r\bigg(n-1,\bigcup\limits_{i=2}^k S^+_{d_i}\bigg).
\end{aligned}
\end{equation*}
Therefore, we can consider a copy of $\bigcup\limits_{i=2}^k S^+_{d_i}$ (that we indicate as $\bigcup\limits_{i=2}^k S^+_{d_i}$ in the sequel) 
after removing $u$ from $\mathcal{H}$, and apply induction on 
the resulting hypergraph.

Next we will show that there is a copy of $S^+_{d_1}$ with center $u$ that is vertex-disjoint from  
$\bigcup\limits_{i=2}^k S^+_{d_i}$. Observe that $u$ and any vertex $v\in V(\bigcup\limits_{i=2}^k S^+_{d_i})$ are contained in at most $\binom{n-2}{r-2}$ common hyperedges. Therefore the number of hyperedges containing $u$ and a vertex of 
 $\bigcup\limits_{i=2}^k S^+_{d_i}$ is at most
\begin{equation*}
\bigcup\limits_{i=2}^k ((r-1)d_i+1)\binom{n-2}{r-2}=O(n^{r-2}).
\end{equation*}
Note that $d(u)=\Omega(n^{r-1})$. Therefore if we remove the hyperedges incident to 
 $\bigcup\limits_{i=2}^k S^+_{d_i}$, there still are $\Omega(n^{r-1})$ hyperedges incident to $u$. As $r\ge 3$ and $d_1\ge 2$, we may apply Theorem \ref{ch1:DE} to these hyperedges to get a copy of $S^+_{d_1}$ that is vertex-disjoint from  
 $\bigcup\limits_{i=2}^k S^+_{d_i}$, {\it i.e.}, $\mathcal{H}$ contains a copy of $\bigcup\limits_{i=1}^k S^+_{d_i}$, a contradiction.

Thus, the theorem holds. \qed

\section{Proof of Theorem \ref{ch1:thm2}}\label{sec-thm2} 

Before proving Theorem \ref{ch1:thm2}, we would like to introduce an average degree lemma obtained by Khormali and Palmer~\cite{KP22} that will be used in 
our proof.
\begin{lemma}[Average Degree Lemma, \cite{KP22}]\label{lem1}
Fix integers $d$ and $\Delta$ and a constant $0\le \varepsilon< 1$. If $\mathcal{H}$ is a hypergraph with average degree at least $d-\varepsilon$ and maximum degree at most $\Delta$, then the number of vertices in $\mathcal{H}$ of degree less than $d$ is at most
\begin{equation*}
\frac{\Delta-d+\varepsilon}{\Delta-d+1}n.
\end{equation*}
In particular, the number of vertices in $\mathcal{H}$ of degree at least $d$ is $\Omega(n)$.
\end{lemma}

Now we are ready to prove Theorem \ref{ch1:thm2}. Let $\mathcal{H}$ be an $r$-uniform $n$-vertex linear hypergraph that is $\bigcup\limits_{i=1}^k S^+_{d_i}$-free. Let $A$ be the vertex set 
each vertex of which has degree at least some fixed large enough constant $D=D(d_1,k,r)$. If $|A|\ge k$, then we can greedily embed a copy of $\bigcup\limits_{i=1}^k S^+_{d_i}$ into $\mathcal{H}$. Thus we have $|A|\le k-1$.

After removing the vertices of $A$ and the hyperedges incident to them from $\mathcal{H}$, we 
 obtain an $r$-uniform hypergraph $\mathcal{H}'$. The maximum degree in $\mathcal{H}'$ is less than $D$. Let $|A|=i-1$. Therefore $1\le i\le k$. If the average degree in $\mathcal{H'}$ is at least $d_i-\varepsilon$ for any $0<\varepsilon<1$, then by 
 Lemma~\ref{lem1} we have $\Omega(n)$ vertices of degree at least $d_i$ in $\mathcal{H'}$.
In this case we can greedily embed a copy of $\bigcup\limits_{t=i}^k S^+_{d_t}$ into $\mathcal{H'}$. 
Then, together with $\bigcup\limits_{t=1}^{i-1} S^+_{d_t}$ consisting of hyperedges incident to $A$, we 
obtain a copy of $\bigcup\limits_{t=1}^k S^+_{d_t}$ in $\mathcal{H}$, a contradiction.

Therefore, the average degree in $\mathcal{H'}$ is at most $d_i-1$. Thus we have
\begin{equation*}
|E(\mathcal{H'})|\le \frac{d_i-1}{r}(n-|A|).
\end{equation*}
Let $B$ be the vertex set of $\mathcal{H'}$, {\it i.e.}, $B=V(\mathcal{H})\setminus A$. Denote by $E(A,B)$ the hyperedge set of $\mathcal{H}$ that contain at least one vertex of $A$ and one vertex of $B$. To count the size of $E(A,B)$, we first count the number of pairs $(h,\{u,v\})$ where $h$ is a hyperedge of $\mathcal{H}$, $u\in A\cap h$ and $v\in B\cap h$.
Let $h\in E(A,B)$. Then we have $|A\cap h|$ choices for $u$ and $|B\cap h|$ choices for $v$. Thus the number of pairs $(h,\{u,v\})$ is
\begin{equation*}
\sum\limits_{h\in E(A,B)} |A\cap h||B\cap h|\ge \sum\limits_{h\in E(A,B)} (r-1)=|E(A,B)|(r-1).
\end{equation*}

On the other hand, for a fixed $u$ and $v$ there is at most one hyperedge containing them, as $\mathcal{H}$ is linear. Thus, the number of pairs $(h, \{u,v\})$ is at most $|A|(n-|A|)$.

Combining these two  
estimates 
for the number of pairs $(h,\{u,v\})$, and solving for $|E(A,B)|$ gives
\begin{equation*}
|E(A,B)|\le \frac{|A|}{r-1}(n-|A|).
\end{equation*}

Furthermore, the maximum number of hyperedges contained completely in $A$ is at most $\binom{|A|}{2}/\binom{r}{2}$, as each pair of vertices in $A$ is contained in at most one hyperedge.

Therefore, the hyperedge number of $\mathcal{H}$ is
\begin{equation*}
|E(\mathcal{H})|\le \frac{d_i-1}{r}(n-i+1)+\frac{i-1}{r-1}(n-i+1)+\frac{\binom{i-1}{2}}{\binom{r}{2}}.
\end{equation*}

As $1\le i\le k$, we have that for $n$ large enough,
\begin{equation*}
|E(\mathcal{H})|\le \mathop{max}\limits_{1\le i\le k} \bigg\{\frac{d_i-1}{r}(n-i+1)+\frac{i-1}{r-1}(n-i+1)+\frac{\binom{i-1}{2}}{\binom{r}{2}}\bigg\}.
\end{equation*}

Just like the construction of the 
extremal linear hypergraph that Khormali and Palmer presented 
for the sharpness statement 
in Theorem~\ref{ch1:KP2}, we can present a construction that satisfies the sharpness assertion.
Denote by $[r]^d$ the integer lattice formed by $d$-tuples from the set $[r]$. Indeed, we can view $[r]^d$ as a hypergraph $\mathcal{H}$, as follows. The vertex set of $\mathcal{H}$ 
consists of all $d$-tuples from the set $[r]$. And the collection of $d$-tuples that are fixed in all but one coordinate form a hyperedge. Thus, $[r]^d$ is an $r^d$-vertex $r$-uniform hypergraph with $d\cdot r^{d-1}$ hyperedges.
It is easy to check that $[r]^d$ is linear, as any two hyperedges of
it are either disjoint or intersect in exactly one vertex. Furthermore, $[r]^d$ is $d$-regular, as every vertex of it 
is included in exactly $d$ hyperedges. Note that the hyperedges of $[r]^d$ can be partitioned into $d$ classes, each of which forms a matching. This gives a natural proper hyperedge-coloring of $\mathcal{H}$. We call such a proper hyperedge-coloring a canonical coloring.

The Cartesian product of two sets $A$ and $B$, denoted $A\times B$, is the set of all ordered pairs $(a,b)$ where $a\in A$ and $b\in B$. Let $\mathcal{H}$ and $\mathcal{G}$ be two hypergraphs. Then the Cartesian product 
$\mathcal{H}\times \mathcal{G}$ 
of $\mathcal{H}$ and $\mathcal{G}$ 
 is defined as the hypergraph on vertex set $V(\mathcal{H})\times V(\mathcal{G})$ with hyperedge set
\begin{equation*}
E(\mathcal{H}\times \mathcal{G})=\{\{u\}\times e|u\in V(\mathcal{H}),e\in E(\mathcal{G})\}\cup \{f\times \{v\}|v\in V(\mathcal{G}),f\in E(\mathcal{H})\}.
\end{equation*}

Note that if $\mathcal{H}$ is $r$-uniform and $\mathcal{G}$ is $s$-uniform, then the hyperedges in $\{f\times \{v\} 
 |v\in V(\mathcal{G}),f\in E(\mathcal{H})\}$ are of size $r$. We claim that if two hypergraphs $\mathcal{H}$ and $\mathcal{G}$ are both linear, then $\mathcal{H}\times \mathcal{G}$ is linear.
Indeed, for any two vertices $u_1\times v_1, u_2\times v_2\in V(\mathcal{H})\times V(\mathcal{G})$, they can be contained in one hyperedge if either $u_1=u_2$ or $v_1=v_2$. Assume that $u_1=u_2=u$. Then we note that $u\times v_1$ and $u\times v_2$ are contained in a hyperedge if and only if $v_1$ and $v_2$ are contained in a hyperedge in $\mathcal{G}$.

Now suppose that $n-i+1$ is divisible by $(r-1)^{i-1}$ and $r^{d_i-1}$ for any $1\le i\le k$. We can construct a hypergraph $\mathcal{H}^*_i$ as follows. The vertex set of $\mathcal{H}^*_i$ embodies two disjoint vertex subsets $A^*$ and $B^*$, where $A^*=\{a_1,a_2,\ldots,a_{i-1}\}$ and $B^*$ has $n-i+1$ vertices whose vertices are partitioned into distinct copies of $[r-1]^{i-1}\times [r]^{d_i-1}$.
Assume that all hyperedges of size $r-1$ in the copies of $[r-1]^{i-1}\times [r]^{d_i-1}$ inherit their canonical hyperedge-coloring from the original hypergraph $[r-1]^{i-1}$. The hyperedge set of $\mathcal{H}^*_i$ is partitioned into two hyperedge subsets $C^*$ and $D^*$, where $C^*$ consists of all hyperedges of size $r$ in the copies of $[r]^{d_i-1}$ in $B^*$, and $D^*$ consists of every $h\cup \{a_j\}$ where $a_j\in A^*$ and $h$ is a hyperedge of color $j$ in a copy of $[r-1]^{i-1}$ with a canonical hyperedge-coloring. Except for the above two hyperedge sets $C^*$ and $D^*$, we can embed $\binom{i-1}{2}/\binom{r}{2}$ hyperedges into $A^*$ without violating the linear property of $\mathcal{H}$.

We claim that $\mathcal{H}^*_i$ is linear. For any two hyperedges of $C^*$, they are hyperedges from copies of the linear hypergraph $[r]^{d_i-1}$, and therefore intersect in at most one vertex. As for any two hyperedges of $D^*$, they 
 either have the same color and therefore intersect in a vertex in $A^*$ but not in $B^*$, or have 
 different colors and may intersect in $B^*$ in one vertex, but do not intersect in $A^*$. And a hyperedge of $C^*$ and a hyperedge of $D^*$ intersect in at most one vertex by the construction of $[r-1]^{i-1}\times [r]^{d_i-1}$.

Next step is to show that $\mathcal{H}^*_i$ is $\bigcup\limits_{i=1}^k S^+_{d_i}$-free.
Every vertex in $B^*$ is incident to exactly $d_i-1$ hyperedges that are contained completely in $B^*$. Therefore, if there exists a copy of $\bigcup\limits_{j=1}^k S^+_{d_j}$ in $\mathcal{H}^*_i$, then it must use $i$ distinct vertices in $A^*$. Since $A^*$ has only $i-1$ vertices, we have that $\mathcal{H}^*_i$ is $\bigcup\limits_{i=1}^k S^+_{d_i}$-free.

Fix $1\le i\le k$. The number of hyperedge in $\mathcal{H}^*_i$ is
\begin{equation*}
\begin{aligned}
|E(\mathcal{H}^*_i)|&=|E(A^*,B^*)|+|E(A^*)|+|E(B^*)|\\
&\ge (i-1)(r-1)^{i-2}\cdot \frac{n-i+1}{(r-1)^{i-1}}+\frac{\binom{i-1}{2}}{\binom{r}{2}}+(d_i-1)r^{d_i-2}\cdot \frac{n-i+1}{r^{d_i-1}}\\
&=\bigg(\frac{d_i-1}{r}+\frac{i-1}{r-1}\bigg)(n-i+1)+\frac{\binom{i-1}{2}}{\binom{r}{2}}.
\end{aligned}
\end{equation*}

Thus among all $1\le i\le k$, we can get an extremal hypergraph $\mathcal{H}^*$ such that
\begin{equation*}
\begin{aligned}
|E(\mathcal{H}^*)|&=\mathop{max}\limits_{1\le i\le k} \{|E(\mathcal{H}^*_i)|\}\\
&\ge \mathop{max}\limits_{1\le i\le k} \bigg\{\bigg(\frac{d_i-1}{r}+\frac{i-1}{r-1}\bigg)(n-i+1)+\frac{\binom{i-1}{2}}{\binom{r}{2}}\bigg\}.
\end{aligned}
\end{equation*} \qed

\section{Proof of Theorem \ref{ch1:thm3}}\label{sec-thm3} 
Let $G$ be a graph, $\mathcal{G}$ a Berge-$G$ and $f$ an arbitrary bijection from $\mathcal{G}$ to $G$ such that $f(h)\subset h$ for any $h\in E(\mathcal{G})$, {\it i.e.}, $f$ is a bijection that establishes that $\mathcal{G}$ is a Berge-$G$. We call the graph formed by the image of $f$ a skeleton of $\mathcal{G}$. Thus, a skeleton of $\mathcal{G}$ is a copy of $G$ embedded into the vertex set of $\mathcal{G}$.

Before we prove Theorem \ref{ch1:thm3}, we first introduce
the following two useful lemmas.

\begin{lemma}[\cite{KP22}]\label{lem2}
Fix integers $r\ge 2, \ell\le r$ and let $\mathcal{G}$ be an $r$-uniform hypergraph. If $v$ is a vertex of degree $d(v)\ge \ell$ in $\mathcal{G}$, then $\mathcal{G}$ contains a Berge-$S_\ell$ with center $v$.
\end{lemma}

\begin{lemma}[\cite{KP22}]\label{lem3}
Fix integers $\ell>r\ge 2$ and let $\mathcal{G}$ be an $r$-uniform hypergraph. If $v$ is a vertex of degree $d(v)> \binom{\ell-1}{r-1}$ in $\mathcal{G}$, then $\mathcal{G}$ contains a Berge-$S_\ell$ with center $v$.
\end{lemma}

Now we are ready to prove Theorem \ref{ch1:thm3}.
For the lower bound, we need to construct an extremal hypergraph. Fix $1\le s\le k-1$. We construct a hypergraph $\mathcal{H}^*_s$ as follows. We consider a $(d_{s+1}-1)$-regular $(r-s)$-uniform hypergraph on a vertex set $B^*$ of size $n-s$. Such a hypergraph exists as $r-s$ divides $n-s$ and $r-s\ge d_{s+1}$. Then we add a fixed vertex set $A^*$ of size $s$ to each hyperedge to form an $r$-uniform hypergraph on $n$ vertices. Each vertex in $B^*$ has degree $d_{s+1}-1$, so the skeleton of any
Berge-$S_{d_j}$ for $1\le j\le s+1$ must use at least one vertex from $A^*$. Therefore, $\mathcal{H}^*_s$ is Berge-$\bigcup\limits_{j=1}^k S_{d_j}$-free and has
\begin{equation*}
|E(\mathcal{H}^*_s)|=\frac{d_{s+1}-1}{r-s}(n-s).
\end{equation*}
Thus, among all $\{\mathcal{H}^*_s\}$ for $1\le s\le k-1$, we can get an extremal hypergraph $\mathcal{H}^*$ such that
\begin{equation*}
|E(\mathcal{H}^*)|=\mathop{max}\limits_{1\le s\le k-1}\bigg\{\frac{d_{s+1}-1}{r-s}(n-s)\bigg\}.
\end{equation*}

Next we prove the upper bound. Let $\mathcal{H}$ be an $n$-vertex Berge-$\bigcup\limits_{i=1}^k S_{d_i}$-free $r$-uniform hypergraph. Let $A$ be the vertex set of $\mathcal{H}$  
each vertex of which has degree greater than some fixed large enough constant $D=D(d_1,k,r)$, and let $B=V(\mathcal{H})\setminus A$. Assume that $|A|=s$. Then we have the following claim.
\begin{claim}
$s\le k-1$.
\end{claim}
\begin{proof}
Suppose to the contrary that $s\ge k$. Let $A'\subseteq A$ be a set of $k$ vertices of degree at least $D$. As $D$ is large enough, it is easy to see that there is a Berge-$S_{d_{k-1}}$ with center in $A'$ but whose skeleton is otherwise disjoint from $A'$.

%%HB: Please check the notation, as some of the unions below should be sums, I think. Apologies for not noticing this before in the earlier version of the paper. Perhaps I missed the same error at other places. Please go through the whole file and check it carefully. 
Now we suppose that there exists a Berge-$\bigcup\limits_{j=1}^{k-1} S_{d_j}$ whose skeleton intersects 
$A'$ in $k-1$ vertices. This Berge-$\bigcup\limits_{j=1}^{k-1} S_{d_j}$ has $\bigcup\limits_{j=1}^{k-1} d_j$ hyperedges and its skeleton spans $\bigcup\limits_{j=1}^{k-1} (d_j+1)$ vertices. Let $u$ be the vertex in $A'$ not in the skeleton of this Berge-$\bigcup\limits_{j=1}^{k-1} d_j$.
Then we remove the hyperedges of the Berge-$\bigcup\limits_{j=1}^{k-1} d_j$ from $\mathcal{H}$.
As $D$ is large enough, by Lemma \ref{lem2} or Lemma \ref{lem3} we have that among the remaining hyperedges there is a Berge-$S_{\cup_{1\le j\le k-1} (d_j+1)+d_{k}}$, denoted $\mathcal{S}$, with center $u$. Thus $\mathcal{S}$ is hyperedge-disjoint from the Berge-$\bigcup\limits_{j=1}^{k-1} S_{d_j}$. However, the skeleton of both of these subhypergraphs may intersect. At most $\bigcup\limits_{j=1}^{k-1} (d_j+1)$ vertices of the skeleton of $\mathcal{S}$
are shared with the skeleton of the Berge-$\bigcup\limits_{j=1}^{k-1} S_{d_j}$. Therefore, there is a Berge-$S_{d_k}$ whose skeleton is disjoint from the skeleton of the Berge-$\bigcup\limits_{j=1}^{k-1} S_{d_j}$. In particular, we have a Berge-$\bigcup\limits_{j=1}^{k} S_{d_j}$, a contradiction. 
This confirms the claim. 
\end{proof}

We also need the following claim before we can complete our proof. 
\begin{claim}\label{claim1}
If $s=k-1$, then every vertex in $B$ has degree at most $d_k-1$. If $s<k-1$, then the average degree of the vertices in $B$ is at most $d_{s+1}-1$.
\end{claim}
\begin{proof}
We first consider the case that $s=k-1$, and suppose that there is a vertex $v$ in $B$ of degree at least $d_k$.
Let us remove vertices from the hyperedges incident to $v$ so that the resulting hyperedges are disjoint from $A$ and of size exactly $d_k$. This is possible as $r\ge d_1+k-1$. By Lemma \ref{lem2}, there is a Berge-$S_{d_k}$ with center $v$ among these $d_k$-uniform hyperedges. The skeleton of this Berge-$S_{d_k}$ is necessarily contained in $B$. It is easy to see that this Berge-$S_{d_k}$ corresponds to a Berge-$S_{d_k}$, denoted by $\mathcal{S}$, in $\mathcal{H}$ whose skeleton is contained in $B$. Since the degrees of the vertices in $A$ are large enough, we can construct a Berge-$\bigcup\limits_{j=1}^{k-1} S_{d_j}$
whose hyperedges and skeleton are disjoint from those of $\mathcal{S}$. Therefore $\mathcal{H}$ contains a Berge-$\bigcup\limits_{j=1}^{k} S_{d_j}$, a contradiction.

Next we consider the case that 
$s<k-1$. Suppose that the average degree of the vertices in $B$ is at least $d_{s+1}-1+\varepsilon$ for some constant $0<\varepsilon\le 1$. Thus the sum of degrees of the vertices in $B$ is at least $(d_{s+1}-1+\varepsilon)(n-s)$.
Denote by $t$ the number of vertices of degree at most $d_{s+1}-1$ in $B$. Recall that the vertices of $B$ have degree at most $D$. Thus the sum of degrees of the vertices in $B$ is at most $(d_{s+1}-1)t+D(n-s-t)$.

Combining these estimates and solving for $t$ gives
\begin{equation*}
t\le \frac{D-d_{s+1}+1-\varepsilon}{D-d_{s+1}+1}(n-s)=(1-\varepsilon')(n-s)
\end{equation*}
for some $\varepsilon'>0$ depending only on $\varepsilon, d_{s+1}, k$ and $r$. Therefore, the number of vertices of degree at least $d_{s+1}$ is $\varepsilon'n=\Omega(n)$.

Let us call a pair of vertices $u,v\in B$ \emph{far} if they do not share a common neighbor in $B$ (they may still have a common neighbor in $A$). Since the vertices of $B$ have constant maximum degree, we can find a subset $B'\subset B$ of size $\Omega(n)$ such that all vertices have degree at least $d_{s+1}$ and all pairs of vertices are far.

For each vertex $u\in B'$, there is a Berge-$S_{d_{s+1}}$ with center $u$. The hyperedges of this Berge-$S_{d_{s+1}}$ may intersect $A$. However, as $r\ge d_1+k-1>d_1+s$ we can find a skeleton of this Berge-$S_{d_{s+1}}$ that does not include a vertex of $A$. Since any two vertices in $B'$ do not share a common neighbor in $B$, we have a collection of pairwise hyperedge-disjoint copies of a Berge-$S_{d_{s+1}}$ whose skeletons are also pairwise vertex-disjoint. Since $n$ is large enough, we can find a copy of a Berge-$\bigcup\limits_{j=s+1}^k S_{d_j}$,
denoted by $\mathcal{S}$. Since the degrees of the vertices in $A$ are large enough, we can construct a Berge-$\bigcup\limits_{j=1}^{s} S_{d_j}$ whose hyperedges and skeleton are disjoint from those of $\mathcal{S}$. Therefore we have a Berge-$\bigcup\limits_{j=1}^{k} S_{d_j}$, a contradiction. 
This confirms our second claim. 
\end{proof}

Now, let us estimate the number of hyperedges in $\mathcal{H}$. As each hyperedge incident to $A$ is counted at most $s=|A|$ times,
\begin{equation*}
\sum\limits_{u\in A} d(u)\le s|E(\mathcal{H})|.
\end{equation*}
Let $d(B)$ be the average degree of the vertices in $B$ and observe that
\begin{equation*}
\begin{aligned}
r|E(\mathcal{H})|&=\sum\limits_{u\in A} d(u)+\sum\limits_{v\in B} d(v)\\
&\le s|E(\mathcal{H})|+d(B)(n-s).
\end{aligned}
\end{equation*}
Solving for $|E(\mathcal{H})|$ gives
\begin{equation*}
|E(\mathcal{H})|\le \frac{d(B)}{r-s}(n-s).
\end{equation*}
By Claim \ref{claim1} we have
\begin{equation*}
|E(\mathcal{H})|\le \mathop{max}\limits_{1\le s\le k-1} \bigg\{\frac{d_{s+1}-1}{r-s}(n-s)\bigg\}.
\end{equation*} \qed

\section{Proof of Theorem \ref{ch1:thm4}}\label{sec-thm4} 
In our proof, we make use of the following known result. 
\begin{lemma}[\cite{GMP20}]
Fix integers $\ell\ge 1$ and $r\ge 3$.
\begin{itemize}
\item [(1)] If $\ell >r$, then
\begin{equation*}
ex_r(n,\mbox{Berge-}S_\ell)\le \binom{\ell}{r}\frac{n}{\ell}.
\end{equation*}
Furthermore, this bound is sharp whenever $\ell$ divides $n$.

\item [(2)] If $\ell \le r$, then
\begin{equation*}
ex_r(n,\mbox{Berge-}S_\ell)\le \frac{\ell-1}{r}n.
\end{equation*}
Furthermore, this bound is sharp  
whenever 
$r$ divides $n$.
\end{itemize}
\end{lemma}
Now we are ready to prove Theorem \ref{ch1:thm4}.
We start again with the lower bounds. 
If $r\le d_k+k-2$, then we can present an extremal construction for 
an $r$-uniform $n$-vertex Berge-$\bigcup\limits_{1\le i\le k} S_{d_i}$-free hypergraph $\mathcal{H}_i$, as follows. Fix $1\le i\le k$. We put $n-i+1=q_i\cdot d_i+t_i$ with $0\le t_i<d_i$. Let $V(\mathcal{H}_i)=A^*_i\cup B^*_i$ with $|A^*_i|=i-1$ and $|B^*_i|=n-i+1$. For the case that $i=k$, we partition the vertices of $B^*_k$ into $q_k$ classes of size $d_k$ and
a single class of size $t_k$ if $t_k>0$. For each partition class $S^*_k$ of $B^*_k$ we form a complete $r$-uniform hypergraph $K^{r}_{d_k+k-1}$ ($K^{r}_{t_k+k-1}$) on the vertices of $A^*_k\cup S^*_k$. Thus we have
\begin{equation*}
|E(\mathcal{H}_k)|=\bigg(\binom{d_k+k-1}{r}-\binom{k-1}{r}\bigg)\frac{n-k+1-t_k}{d_k}+\binom{t_k+k-1}{r}.
\end{equation*}
The skeleton of any Berge-$S_{d_j}$ for $1\le j\le k$ in $\mathcal{H}_k$ must use at least one vertex of $A^*_k$. But there are only $k-1$ vertices in $A^*_k$. Therefore we have $\mathcal{H}_k$ is Berge-$\bigcup\limits_{1\le j\le k} S_{d_j}$-free for any $1\le i\le k$. Thus we can choose $i$ such that $E(\mathcal{H}_i)$ (denoted by $\mathcal{H}^*$)
attains its maximum value, {\it i.e.},
\begin{equation*}
|E(\mathcal{H}^*)|=\mathop{max}\limits_{1\le i\le k}\bigg\{\bigg(\binom{d_i+i-1}{r}-\binom{i-1}{r}\bigg)\frac{n-i+1-t_i}{d_i}+\binom{t_i+i-1}{r}\bigg\}.
\end{equation*}

If $r\ge d_k+k-1$, then we can present an extremal construction for 
an $r$-uniform $n$-vertex Berge-$\bigcup\limits_{1\le i\le k} S_{d_i}$-free hypergraph $\mathcal{H}'_k$, as follows. 
Since $r-k+1$ divides $n-k+1$ and $r-k+1\ge d_k$, we can construct a 
$(d_k-1)$-regular $(r-k+1)$-uniform hypergraph on a vertex set $B^*$ of size $n-k+1$.
Then we add a fixed vertex set $A^*$ of size $k-1$ to each hyperedge to form an $r$-uniform hypergraph on $n$ vertices. Each vertex in $B^*$ has degree $d_k-1$, so the skeleton of any
Berge-$S_{d_j}$ for $1\le j\le k$ must use at least one vertex from $A^*$. Therefore, $\mathcal{H}'_k$ is Berge-$\bigcup\limits_{j=1}^k S_{d_j}$-free and 
\begin{equation*}
|E(\mathcal{H}'_k)|=\frac{d_k-1}{r-k+1}(n-k+1).
\end{equation*}
Thus we can get an extremal hypergraph $\mathcal{H}'$ such that
\begin{equation*}
|E(\mathcal{H}')|=\mathop{max}\limits_{1\le i\le k}\bigg\{\frac{d_i-1}{r-i+1}(n-i+1)\bigg\}.
\end{equation*}

Next we  prove the upper bound. We proceed by induction on $r$. If $r=2$ ($r\le d_i+i-1$ for any $1\le i\le k$), then a Berge-$S_{d_i}$ is simply a copy of the graph $S_{d_i}$ for any $1\le i\le k$.
By Theorem \ref{ch1:LLP}, we have
\begin{equation*}
\begin{aligned}
ex_2\bigg(n,\mbox{Berge}-\bigcup\limits_{1\le i\le k} S_{d_i}\bigg)&=ex\bigg(n,\bigcup\limits_{1\le i\le k} S_{d_i}\bigg)\\
&\le \mathop{max}\limits_{1\le i\le k}\Bigg\{(i-1)(n-i+1)+\binom{i-1}{2}+\frac{d_i-1}{2}(n-i+1)\Bigg\}\\
&=\mathop{max}\limits_{1\le i\le k} \bigg\{\frac{n-i+1}{d_i}\bigg(\binom{d_i+i-1}{2}-\binom{i-1}{2}\bigg)+\binom{i-1}{2}\bigg\}.
\end{aligned}
\end{equation*}

Let $r\ge 3$ and assume that the statement of the theorem holds for $r-1$. Let $\mathcal{H}$ be an $r$-uniform $n$-vertex Berge-$\bigcup\limits_{1\le i\le k} S_{d_i}$-free hypergraph of size at least $|E(\mathcal{H}^*)|$. Let $A$ be the set of vertices in $\mathcal{H}$ of degree larger than some fixed large enough constant $D=D(d_1,k,r)$, and let $B=V(\mathcal{H})\setminus A$. Put $c=|A|$.
We first prove the following claim.
\begin{claim}
$c\le k-1$.
\end{claim}
\begin{proof}
%%HB: The same happens below, I think. Please check it. 
Suppose to the contrary that $c\ge k$ and let $A'\subseteq A$ be a set of $k$ vertices of degree at least $D$. As $D$ is large enough, there must be a Berge-$S_{d_1}$ with center in $A'$ but whose skeleton is otherwise disjoint from $A'$.
Now we suppose that there exists a Berge-$\bigcup\limits_{j=1}^{k-1} S_{d_j}$ whose skeleton intersects 
$A'$ in $k-1$ vertices. This Berge-$\bigcup\limits_{j=1}^{k-1} S_{d_j}$ has $\bigcup\limits_{j=1}^{k-1} d_j$ hyperedges and its skeleton spans $\bigcup\limits_{j=1}^{k-1} (d_j+1)$ vertices. Let $u$ be the vertex in $A'$ not in the skeleton of this Berge-$\bigcup\limits_{j=1}^{k-1} d_j$.
Then we remove the hyperedges of the Berge-$\bigcup\limits_{j=1}^{k-1} d_j$ from $\mathcal{H}$.
As $D$ is large enough, by Lemma \ref{lem3} we have that among the remaining hyperedges there is a Berge-$S_{\cup_{1\le j\le k-1} (d_j+1)+d_k}$, denoted by $\mathcal{S}$, with center $u$. Thus $\mathcal{S}$ is hyperedge-disjoint from the Berge-$\bigcup\limits_{j=1}^{k-1} S_{d_j}$. However, the skeleton of both of these subhypergraphs may intersect. At most $\bigcup\limits_{j=1}^{k-1} (d_j+1)$ vertices of the skeleton of $\mathcal{S}$
are shared with the skeleton of the Berge-$\bigcup\limits_{j=1}^{k-1} S_{d_j}$. Therefore, there is a Berge-$S_{d_k}$ whose skeleton is disjoint from the skeleton of the Berge-$\bigcup\limits_{j=1}^{k-1} S_{d_j}$. In particular, we have a Berge-$\bigcup\limits_{j=1}^{k} S_{d_j}$, a contradiction.
\end{proof}

Next we distinguish two cases.

\medskip\noindent
{\bf Case 1. $c=k-1$.}

We will discuss two subcases.

\medskip\noindent
{\bf Subcase 1.1. $r\le d_k+k-2$.}

We claim that each vertex of $B$ has degree at most $\binom{d_k+k-2}{r-1}$. Otherwise, by Lemma \ref{lem3} we have that there is a Berge-$S_{d_k+k-1}$ with center $x$ which has degree greater than $\binom{d_k+k-2}{r-1}$. The skeleton of this Berge-$S_{d_k+k-1}$ uses at most $k-1$ vertices from $A$, so there remains a Berge-$S_{d_k}$, denoted by $\mathcal{S}$, whose skeleton is contained in $B$. As $D$ is large enough, we may construct $k-1$ more pairwise hyperedge-disjoint copies of a Berge-$S_{d_1}$ (with pairwise vertex-disjoint skeletons) that are hyperedge-disjoint from $\mathcal{S}$ and whose skeletons are vertex-disjoint
from the skeleton of $\mathcal{S}$. Thus, we have a copy of a Berge-$\bigcup\limits_{1\le i\le k} S_{d_i}$ in $\mathcal{H}$, a contradiction.

Now we compare $\mathcal{H}$ with $\mathcal{H}_k$. Firstly, by definition we have $|E(\mathcal{H})|\ge |E(\mathcal{H}^*)|\ge |E(\mathcal{H}_k)|$. Suppose that there are two vertices $y\in A$ and $y^*\in A^*_k$ such that $d_{\mathcal{H}}(y)>d_{\mathcal{H}_k}(y^*)$.
For each $1\le j\le r-1$, we define
\begin{equation*}
E^y_{j}=\{e\setminus A|y\in e\in E(\mathcal{H}), |e\setminus A|=j\}
\end{equation*}
and
\begin{equation*}
E^{y^*}_{j}=\{f\setminus A^*|y^*\in f\in E(\mathcal{H}_k), |f\setminus A^*|=j\}.
\end{equation*}
It is easy to see that the members of $E^y_j$ have multiplicity at most $\binom{k-2}{r-1-j}$ and those in $E^{y^*}_j$ 
have multiplicity exactly $\binom{k-2}{r-1-j}$.

If $k-1\ge r-1$, then by the construction of $\mathcal{H}_k$ we have that each vertex of $B^*_k$ is in a hyperedge with each subset of $A^*_k$ of size $r-1$. This implies that $|E^{y^*}_1|\ge |E^y_1|$. When $k-1<r-1$, then $E^{y^*}_1=E^y_1=\emptyset$. Therefore, since $d_{\mathcal{H}}(y)\ge d_{\mathcal{H}_k}(y^*)$, we have that $|E^y_j|>|E^{y^*}_j|$ for some $j\ge 2$. Let $\mathcal{G}$ be the $j$-uniform hypergraph resulting from deleting all repeated hyperedges in this $E^y_j$. Thus, the number of hyperedges in the $j$-uniform hypergraph $\mathcal{G}$ can be calculated as
\begin{equation*}
\begin{aligned}
|E(\mathcal{G})|&\ge \binom{k-2}{r-1-j}^{-1}|E^y_j|\\
&>\binom{k-2}{r-1-j}^{-1}|E^{y^*}_j|\\
&=\binom{k-2}{r-1-j}^{-1}\binom{k-2}{r-1-j}\frac{n-k+1}{d_k}\binom{d_k}{j}\\
&=\frac{n-k+1}{d_k}\binom{d_k}{j}\\
&\ge ex_j(n-k+1,\mbox{Berge-}S_{d_k}).
\end{aligned}
\end{equation*}

Therefore, there is a $j$-uniform Berge-$S_{d_k}$ in $\mathcal{G}$ on the vertices of $B$. As each hyperedge of $\mathcal{G}$ is contained in a hyperedge of $\mathcal{H}$, this corresponds to a Berge-$S_{d_k}$ in $\mathcal{H}$ whose skeleton is contained in $B$. As before, the degree condition on the vertices in $A$ guarantees the existence of a Berge-$\bigcup\limits_{1\le j\le k-1} S_{d_j}$
that together with this Berge-$S_{d_k}$ forms a Berge-$\bigcup\limits_{1\le j\le k} S_{d_j}$ in $\mathcal{H}$, a contradiction.

Therefore, every vertex $y\in A$ has degree at most that of the vertices in $A^*_k$ in the construction $\mathcal{H}_k$. Every vertex in $B$ has degree at most
$\binom{d_k+k-2}{r-1}$ while every vertex in $B^*_k$ has degree exactly $\binom{d_k+k-2}{r-1}$.
Thus, $|E(\mathcal{H})|\le |E(\mathcal{H}_k)|\le |E(\mathcal{H}^*)|$.

\medskip\noindent
{\bf Subcase 1.2. $r\ge d_k+k-1$.}

We claim that each vertex of $B$ has degree at most $d_k-1$.
Suppose that there is a vertex $v$ in $B$ of degree at least $d_k$. Let us remove vertices from the hyperedges incident to $v$ so that the resulting hyperedges are disjoint from $A$ and of size exactly $d_k$. This is possible as $r\ge d_k+k-1$. By Lemma \ref{lem2}, there is a Berge-$S_{d_k}$ with center $v$ among these $d_k$-uniform hyperedges. The skeleton of this Berge-$S_{d_k}$ is necessarily contained in $B$. It is easy to see that this Berge-$S_{d_k}$ corresponds to a Berge-$S_{d_k}$, denoted by $\mathcal{S}$, in $\mathcal{H}$ whose skeleton is contained in $B$.
Since the degrees of the vertices in $A$ are large enough, we can construct a Berge-$\bigcup\limits_{j=1}^{k-1} S_{d_j}$ whose hyperedges and skeleton are disjoint from those of $\mathcal{S}$. Therefore $\mathcal{H}$ contains a Berge-$\bigcup\limits_{j=1}^{k} S_{d_j}$, a contradiction.
Thus we have
\begin{equation*}
\begin{aligned}
r|E(\mathcal{H})|&=\sum\limits_{u\in A} d(u)+\sum\limits_{v\in B} d(v)\\
&\le (k-1)|E(\mathcal{H})|+(d_k-1)(n-k+1).
\end{aligned}
\end{equation*}
Solving for $|E(\mathcal{H})|$ gives
\begin{equation*}
|E(\mathcal{H})|\le \frac{d_k-1}{r-k+1}(n-k+1)\le E(\mathcal{H}').
\end{equation*}

\medskip\noindent
{\bf Case 2. $c<k-1$.}

If $r\le d_c+c$, then we claim that $d(B)\le \binom{d_c+c}{r-1}$, where $d(B)$ is the average degree of the vertices in $B$. Otherwise, we have $d(B)\ge \binom{d_c+c}{r-1}+\varepsilon$ for a constant $0<\varepsilon<1$. Let $s$ be the number of vertices of degree at most $\binom{d_c+c}{r-1}$ in $B$. Recall that the vertices in $B$ have degree at most $D$. Thus, the sum of degrees in $B$ is at most $\binom{d_c+c}{r-1}s+D(n-c-s)$. Combining these estimates and solving for $s$ gives
\begin{equation*}
s\le \frac{D-\binom{d_c+c}{r-1}-\varepsilon}{D-\binom{d_c+c}{r-1}}(n-c)=(1-\varepsilon')(n-c)
\end{equation*}
for some $\varepsilon'>0$ not depending on $n$. Therefore, the number of vertices of degree greater than $\binom{d_c+c}{r-1}$ is $\varepsilon' n=\Omega(n)$.

Let us call a pair of vertices $u,v\in B$ \textit{far} if they do not share a common neighbor in $B$ (they may still have a common neighbor in $A$). Since the vertices of $B$ have constant maximum degree we can find a subset $B'$ of size $\Omega(n)$ such that all vertices have degree greater than $\binom{d_c+c}{r-1}$ and all pairs of vertices are far.

For each vertex $u\in B'$, by Lemma \ref{lem3} we have that there is a Berge-$S_{d_c+c+1}$ with center $u$ whose skeleton is disjoint from $A$. Since any two vertices of $B'$ do not share a common neighbor in $B$, we have a collection of hyperedge-disjoint copies of a Berge-$S_{d_c}$. As $n$ is large enough, we can find a Berge-$\bigcup\limits_{c+1\le j\le k} S_{d_j}$ whose skeleton is disjoint from $A$. And since the degrees of the vertices in $A$ are large enough, we can construct a Berge-$\bigcup\limits_{1\le j\le c} S_{d_j}$ whose hyperedges and skeleton are disjoint from those of Berge-$\bigcup\limits_{c+1\le j\le k} S_{d_j}$.
Therefore, $\mathcal{H}$ contains a Berge-$\bigcup\limits_{1\le i\le k} S_{d_i}$, a contradiction.

Observe that $\sum\limits_{x\in A} d(x)\le c|E(\mathcal{H})|$ and
\begin{equation*}
r|E(\mathcal{H})|=\sum\limits_{x\in B} d(x)+\sum\limits_{x\in A} d(x)\le d(B)(n-c)+c|E(\mathcal{H})|.
\end{equation*}
Solving for $|E(\mathcal{H})|$ gives
\begin{equation*}
|E(\mathcal{H})|\le \frac{\binom{d_c+c}{r-1}}{r-c}(n-c).
\end{equation*}
Since $c<k-1$,
\begin{equation*}
|E(\mathcal{H})|\le \mathop{max}\limits_{1\le c<k-1}\bigg\{\frac{\binom{d_c+c}{r-1}}{r-c}(n-c)\bigg\}.
\end{equation*}

Similarly, if $r\ge d_c+c+1$ then we have $d(B)\le d_c+c$. In this case, we have that
\begin{equation*}
|E(\mathcal{H})|\le \mathop{max}\limits_{1\le c<k-1}\bigg\{\frac{d_c+c}{r-c}(n-c)\bigg\}.
\end{equation*} \qed

\section{Concluding Remarks}
We complete this paper with two consequences of our results for 
 generalized Tur\'{a}n numbers. 
 Given two graphs $H$ and $F$, the generalized Tur\'{a}n number of $F$, 
 denoted by $ex(n,H,F)$, 
  is defined as the maximum number of copies of a subgraph $H$ in an $n$-vertex $F$-free graph. 
  These generalized Tur\'{a}n numbers 
  have been studied by different groups of researchers for 
  many types of graphs. 
  We refer the interested readers to  \cite{AS16,HHL23,MQ20,W20,ZWZ22,ZC22} 
  for more details. 

Let $G$ be an 
$n$-vertex 
$F$-free graph. We define an $r$-uniform hypergraph $\mathcal{H}$ with the same vertex set as $G$, 
such that 
an $r$-set in $\mathcal{H}$ is a hyperedge if and only if it is the vertex set of a $K_r$ in $G$. Note that 
this implies that 
$\mathcal{H}$ contains no Berge-$F$ if $G$ is $F$-free. Thus,
\begin{equation}\label{eq1}
ex(n,K_r,F)\le ex_r(n,\mbox{Berge-}F).
\end{equation}

For the case that 
$r\ge d_1+k-1$, by Theorem \ref{ch1:thm3} and (\ref{eq1}) we 
obtain the following corollary. 
\begin{corollary}
Fix integers $d_1\ge d_2\ge \ldots \ge d_k\ge 2$ and $k\ge 2$. If $r\ge d_1+k-1$, then for $n$ large enough,
\begin{equation*}
ex\bigg(n,K_r,\mbox{Berge-}\bigcup\limits_{1\le i\le k}S_{d_i}\bigg)\le \mathop{max}\limits_{1\le i\le k}\bigg\{\frac{d_i-1}{r-i+1}(n-i+1)\bigg\}.
\end{equation*}
\end{corollary}

Similarly, for the case that 
$r\le d_1+k-1$, by Theorem \ref{ch1:thm4} and (\ref{eq1}) we 
obtain the following corollary.
\begin{corollary}
Fix integers $d_1\ge d_2\ge \ldots \ge d_k\ge 2$ and $k\ge 2$. If $r\le d_1+k-2$, then for $n$ large enough,
\begin{equation*}
\begin{aligned}
ex\bigg(n,K_r,\mbox{Berge-}\bigcup\limits_{1\le i\le k}S_{d_i}\bigg)&\le \mathop{max}\bigg\{\mathop{max}\limits_{1\le i\le k}\bigg\{\bigg(\binom{d_i+i-1}{r}-\binom{i-1}{r}\bigg)\left\lceil\frac{n-i+1}{d_i}\right\rceil+\binom{i-1}{r},\\
&\frac{d_i-1}{r-i+1}(n-i+1)\bigg\},\mathop{max}\limits_{1\le c<k-1}\bigg\{\frac{\mathop{max}\Big\{\binom{d_c+c}{r-1},d_c+c\Big\}}{r-c}(n-c)\bigg\}\bigg\}.
\end{aligned}
\end{equation*}
\end{corollary}

\end{document}